\documentclass{amsart}

\usepackage [latin1] {inputenc}

\usepackage {mathrsfs, amsmath,amsthm,  amssymb}
\usepackage{dsfont}
\usepackage [intlimits] {empheq}
\usepackage{paralist}

\newtheorem{thm}{Theorem}[section]
\newtheorem{prop}[thm]{Proposition}
\newtheorem{defn}[thm]{Definition}
\newtheorem{lem}[thm]{Lemma}
\newtheorem{cor}[thm]{Corollary}

\renewcommand{\leq}{\leqslant}
\renewcommand{\geq}{\geqslant}

\title[Application of Character Estimates]{Application of Character Estimates to the Number of $T_2$-Systems of the Alternating Group}
\author{Stefan-Christoph Virchow}

\begin{document}
\begin{abstract} We use character theory and character estimates to show that the number of $T_2$-systems of $A_n$ is at least 
\begin{equation*}
\frac{1}{8n\sqrt{3}}\exp\left(\frac{2\pi}{\sqrt{6}}n^{1/2}\right)(1+o(1)). 
\end{equation*}
Applying  this result, we obtain a lower bound for the number of connected components of the product replacement graph $\Gamma_2(A_n)$.
\end{abstract}

\maketitle

\section{Introduction}
Let $G$ be a finite group and let $d(G)$ be the minimal number of generators of $G$. Fix some integer $k\geq d(G)$ and denote by

\begin{equation*}
\mathcal{N}_k(G):=\{(g_1, \ldots, g_k)\in G^k:~ \langle g_1, \ldots, g_k\rangle=G\}
\end{equation*}
the set of generating $k$-tuples of $G$. We identify $\mathcal{N}_k(G)$ with the set of epimorphisms $\operatorname{Epi}(F_k\twoheadrightarrow G)$, where $F_k$ is the free group on $k$ generators. Now consider the following group action
\begin{align*}
\bigl(\operatorname{Aut}(F_k)\times\ \operatorname {Aut}(G)\bigr)\times\operatorname{Epi}(F_k\twoheadrightarrow G)&\to \operatorname{Epi}(F_k\twoheadrightarrow G)\\
\bigl((\tau,\sigma),\phi\bigr) &\mapsto\sigma\circ\phi\circ\tau^{-1}.
\end{align*}
\textit{Systems of transitivity} or short \textit{$T_k$-systems} are defined to be the orbits of this group action. Denote by $\tau_k(G)$ the number of $T_k$-systems of $G$. B. H. Neumann and H. Neuman \cite{Ne1} introduced $T_k$-systems in 1950. Since then they have been widely investigated by Dunwoody \cite{Du1}, Evans \cite{Ev1}, Gilman \cite{Gi1}, Guralnick and Pak \cite{Gu1}, Neumann \cite{Ne2} and Pak \cite{Pa1}. \par
There has been renewed interest in $T_k$-systems in the recent years because they can be applied to the \textit{product replacement algorithm (PRA).} The PRA is a practical algorithm to generate random elements of finite groups. Celler, Leedham-Green, Murray, Niemeyer and O'Brien \cite{Ce1} introduced the PRA in 1995 and since then it has been widely studied (see \cite{Ba1}, \cite{Ga1}, \cite{Lu1}, \cite{Pa1}). The PRA is defined as follows: Let $(g_1, \ldots, g_k)\in\mathcal{N}_k(G)$ be a generating $k$-tuple of $G$. Define a move to another generating $k$-tuple in the following way: First select uniformly a pair  $(i,j)$ with $1\leq i\neq j\leq k$. Secondly, apply one of the following four operations with equal probability:
\begin{align*}
R_{i,j}^{\pm}&\colon (g_1, \ldots,g_i,\ldots,g_k)\to (g_1, \ldots,g_i\cdot g_j^{\pm 1},\ldots,g_k)\\
L_{i,j}^{\pm}&\colon (g_1, \ldots,g_i,\ldots,g_k)\to (g_1, \ldots,g_j^{\pm 1}\cdot g_i,\ldots,g_k).
\end{align*}
We apply the above moves several times (the choice of the moves must be uniform and independent at each step). Finally, we return a random component of the resulting generating $k$-tuple. This gives us the wanted ``random'' element of the group $G$.\\
The \textit{product replacement graph (PRA graph)} $\Gamma_k(G)$ is a graph with vertices corresponding to generating $k$-tuples of $G$ and edges corresponding to the moves $R_{i,j}^{\pm},~ L_{i,j}^{\pm}$. The PRA can be described as a random walk on this graph. Denote by $\varkappa_k(G)$ the number of connected components of $\Gamma_k(G)$.\par
The connection between the PRA and $T_k$-systems of $G$ is the following: Nielsen showed in \cite{Ni1} that $\operatorname {Aut}(F_k)$ is generated by the \textit{Nielsen moves}
\begin{align*}
R_{i,j}&\colon (x_1, \ldots,x_i,\ldots,x_k)\to (x_1, \ldots,x_i\cdot x_j,\ldots,x_k)\\
L_{i,j}&\colon (x_1, \ldots,x_i,\ldots,x_k)\to (x_1, \ldots,x_j\cdot x_i,\ldots,x_k)\\
P_{i,j}&\colon (x_1, \ldots,x_i,\ldots, x_j,\ldots,x_k)\to (x_1, \ldots,x_j,\ldots, x_i,\ldots,x_k)\\
I_{i}&\colon (x_1, \ldots,x_i,\ldots,x_k)\to (x_1, \ldots,x_i^{-1},\ldots,x_k)
\end{align*}
for $1\leq i\neq j\leq k$. Therefore, we can define a graph with vertices corresponding to $\mathcal{N}_k(G)$ and edges corresponding to the Nielsen moves and to the automorphisms of $G$. The connected components of this graph are exactly the $T_k$-systems of $G$. Now it follows immediately from the definitions that
\begin{equation}
\label{N-2-1-1}
\tau_k(G)\leq\varkappa_k(G).
\end{equation}\par
Little is known about the problem to estimate $\tau_k(G)$ as a function of $k$ and $G$. Of special interest are $T_k$-systems of a finite simple group $G$. A conjecture attributed to Wiegold states that $\tau_k(G)=1$ for $k\geq 3$ (see \cite[Conjecture 2.5.4]{Pa1}). Cooperman and Pak \cite{Co1}, David \cite{Da1}, Evans \cite{Ev1}, Garion \cite{Ga3} and Gilman \cite{Gi1} proved this conjecture for some families of simple groups. The case $k=2$, however, appears to be different. In 2009, S. Garion and A. Shalev \cite[Theorem 1.8]{Ga2} showed that $\tau_2(G)\to\infty$ as $|G|\to\infty$, where $G$ is a finite simple group. Hence, they confirmed a conjecture of Guralnick and Pak \cite{Gu1}. In addition they proved that
\begin{equation*}
\tau_2(A_n)\geq n^{(\frac{1}{2}-\epsilon)\log n}.
\end{equation*}
In 1985, M. J. Evans already established in his Ph.D. Thesis \cite{Ev2} that
\begin{align*}
\tau_2(A_{2n+4})&\geq\#\{C:~C \text{ is conjugacy class of } S_n\}\\
\tau_2(A_{2n+5})&\geq\#\{C:~C \text{ is conjugacy class of } S_n\}
\end{align*}
for $n\geq 2$. However, this thesis is hard to find and difficult to access.
\par
J.-C. Schlage-Puchta \cite{Sc1} applied character theory and character estimates to statistical problems for the symmetric group. Our aim is to show that these methods and concepts can be adopted to establish a lower bound for the number of $T_2$-systems of $A_n$. We claim the following:
\begin{thm}
\label{2-1-1}
Let $n$ be sufficiently large. Then we have
\begin{equation*}
\tau_2(A_n)\geq \frac{1}{8 n\sqrt{3}}\exp\left(\frac{2\pi}{\sqrt{6}}n^{1/2}\right)(1+o(1)).
\end{equation*}
\end{thm}

\noindent
Taking (\ref{N-2-1-1}) into account, this Theorem has the immediate

\begin{cor}
Let $n$ be sufficiently large. Then we have
\begin{equation*}
\varkappa_2(A_n)\geq \frac{1}{8 n\sqrt{3}}\exp\left(\frac{2\pi}{\sqrt{6}}n^{1/2}\right)(1+o(1)).
\end{equation*}
\end{cor}

\section{Proof of Theorem \ref{2-1-1}}

At first we will show, that the number of $T_2$-systems of $A_n$ is bounded below by the number of conjugacy classes of $S_n$ which fulfill certain conditions. Secondly, we will count these conjugacy classes.\par
We start with a Lemma due to Higman (cf. \cite{Ne2}):

\begin{lem}
\label{2-2-1}
Let $G$ be a finite group with $d(G)\leq 2$. Then for each $T_2$-system $S$ of $G$ the set of Aut$(G)$-conjugates of $\{[g_1,g_2]^{\pm 1}\}$ is invariant for all $(g_1,g_2)\in S$.
\end{lem}

\noindent
Higman's Lemma leads us to

\begin{lem}
\label{2-2-2}
Denote $J:=\{[\pi,\sigma]:~(\pi,\sigma)\in\mathcal{N}_2(A_n)\}$. For all sufficiently large $n$ we have
\begin{equation*}
\tau_2(A_n)\geq \#\{C:~C \text{ is a conjugacy class of }S_n\text{ with } C\cap J\neq\emptyset\}.
\end{equation*}
\end{lem}

\begin{proof} Note Aut$(A_n)=\{\alpha|_{A_n}:~~\alpha\in\operatorname{Inn}(S_n)\}$ for $n> 6$ (cf. \cite[Theorem 8.2A]{Di1}). Therefore
\begin{align*}
\mathcal{B}:=&\{K:~K \text{ is } \operatorname{Aut}(A_n)\text{-conjugacy class with } K\cap J\neq\emptyset\}\\
=&\{C:~C \text{ is a conjugacy class of }S_n\text{ with } C\cap J\neq\emptyset\}.
\end{align*}
We have $d(A_n)\leq 2$ (cf. \cite{Di2}). Hence Lemma \ref{2-2-1} yields
\begin{equation*}
\tau_2(A_n)\geq \#\{K\cup K^{-1}: K\in\mathcal{B}\}=|\mathcal{B}|
\end{equation*}
as desired.
\end{proof}

\begin{defn}
Let $C$ be a conjugacy class of $S_n$. Denote by $f(C)$ the number of fixed points of $C$.
\end{defn}
\noindent
We state the following basic result:

\begin{thm}
\label{2-2-3}
Let $p$ denote a prime with $\frac{1}{2}n<p\leq \frac{3}{5}n$. Let $\pi\in A_n$ be a fixed permutation with cycle type 
\begin{equation*}
\lambda =
\begin{cases}
(p,~n-p-2,~1,~1), &\text{if } ~n\equiv 0(2)\\
(p,~n-p-2,~2), &\text{if } ~n\equiv 1(2).
\end{cases}
\end{equation*}
In addition suppose that $C\subset A_n$ is a conjugacy class of $S_n$ such that
\begin{equation*}
f(C)\leq \delta\cdot n\qquad \text{ for } \delta\in(0,\tfrac{1}{4}].
\end{equation*}
Then we obtain for all sufficiently large $n$
\begin{equation*}
A:=\#\{\sigma\in A_n:~[\pi,\sigma]\in C \wedge \langle\pi,\sigma\rangle=A_n\}=|C|\left(1+\mathcal{O}(\delta)\right).
\end{equation*}
In particular, we have $A>0$ for $\delta>0$ sufficiently small.
\end{thm}
\noindent
We will prove this Theorem in the next sections. In view of the previous conclusions, our problem results in counting conjugacy classes of $S_n$ or partitions of $n$. 

\begin{defn}
A \emph{partition} of $n$ is a sequence $\lambda=(\lambda_1,\lambda_2,\ldots,\lambda_l)$ of positive integers such that $\lambda_1\geq\lambda_2\geq\ldots\geq\lambda_l$ and $\lambda_1+\lambda_2+\ldots+\lambda_l=n$. We write $\lambda\vdash n$ to indicate that $\lambda$ is a partition of $n$.\\
Denote by $P(n)$ the number of partitions of $n$.\\
Suppose $\lambda=(\lambda_1,\lambda_2,\ldots,\lambda_l)\vdash n$. The \emph{Ferrers diagram} of $\lambda$ is an array of $n$ boxes having $l$ left-justified rows with row $i$ containing $\lambda_i$ boxes for $1\leq i\leq l$.
\end{defn}
\noindent
Hardy and Ramanujan \cite{Ha1} achieved the subsequent asymptotic formula:

\begin{lem}
\label{2-2-4}
We have
\begin{equation*}
P(n)=\frac{1}{4n\sqrt{3}}\exp\left(\frac{2\pi}{\sqrt{6}}n^{1/2}\right)(1+o(1)).
\end{equation*}
\end{lem}
\noindent
Now, we are able to give the

\begin{proof}[Proof of Theorem \ref{2-1-1}] At first, we note
\begin{align*}
\#\{\lambda\vdash n:~ &\text{number of even parts in }\lambda\text{ is odd}\}\\
=\#\{&\lambda\vdash n:~ \text{number of even parts in }\lambda\text{ is even}\}\\
&-\#\{\lambda\vdash n:~\lambda\text{ consists of distinct odd parts}\}
\end{align*}
(see \cite[Exercise 44]{An1}). Therefore, it follows
\begin{equation*}
\#\{C\subset A_n:~ C\text{ is a conjugacy class of }S_n\}\geq \frac{1}{2}\cdot P(n).
\end{equation*}
Secondly, choose $\delta>0$ sufficiently small. By Lemma \ref{2-2-2}, Theorem \ref{2-2-3} and the above inequality we obtain
\begin{align*}
\tau_2(A_n)&\geq \#\{C\subset A_n:~ C\text{ is a conjugacy class of } S_n\text{ with } f(C)\leq\delta \cdot n\}\\
&\geq \frac{1}{2}P(n)-P(n-\lceil \delta n\rceil).
\end{align*}
Applying Lemma \ref{2-2-4} yields our assertion.
\end{proof}

\section{Character Theory}
The rest of this paper is devoted to the proof of Theorem \ref{2-2-3}. In this section we review some results from character theory which are essential for our proof. We denote by $\operatorname{Irr}(S_n)$ the set of irreducible characters of $S_n$. For a conjugacy class $C$ of $S_n$ and $\chi\in\operatorname{Irr}(S_n)$ we write $\chi(C)$ to denote $\chi(\pi)$ for $\pi\in C$. 

\begin{lem}
\label{2-3-1}
Let $C_1$ and $C_2$ be conjugacy classes of $S_n$ and let $\tau\in S_n$. Then
\begin{equation*}
\#\{(x,y)\in C_1\times C_2:~xy=\tau\}=\frac{|C_1||C_2|}{n!}\sum_{\chi\in \operatorname{Irr}(S_n)}\frac{\chi(C_1)\chi(C_2)\chi(\tau^{-1})}{\chi (1)}.
\end{equation*}
\end{lem}
\begin{proof} See \cite[Proposition 9.33]{Cu1} or \cite[Theorem 6.3.1]{Ke1}.\end{proof}

Recall that the irreducible characters of $S_n$ are explicitly parametrized by the partitions of $n$. Denote by $\chi^{\lambda}$ the irreducible character of $S_n$ corresponding to the partition $\lambda$ of $n$.\par
We shall frequently apply the \textit{Murnaghan-Nakayama rule.} 
\begin{defn}
Let $\lambda\vdash n$ be a partition. A \emph{rim hook} $h$ is an edgewise connected part of the Ferrers diagram of $\lambda$, obtained by starting from a box at the right end of a row and at each step moving upwards or rightwards only, which can be removed to leave a proper  Ferrers diagram denoted by $\lambda\backslash h$. A $r$\emph{-rim hook} is a rim hook containing $r$ boxes.\\
The \emph{leg length} of a rim hook $h$ is
\begin{equation*}
ll(h):=(\text{the number of rows of }h)-1.
\end{equation*}
Let $\pi\in S_n$ be a permutation with cycle type $(1^{\alpha_1},\ldots,q^{\alpha_q},\ldots,n^{\alpha_n})$ and $\alpha_q\geq 1$. Denote $\pi\backslash q\in S_{n-q}$ a permutation with cycle type $(1^{\alpha_1},\ldots,q^{\alpha_q-1},\ldots,(n-q)^{\alpha_{n-q}})$.
\end{defn}

\begin{lem}[Murnaghan-Nakayama Rule]
\label{2-3-2}
Let $\lambda\vdash n$ be a partition. Suppose that $\pi\in S_n$ is a permutation which contains a $q$-cycle. Then we have
\begin{equation*}
\chi^{\lambda}(\pi)=\sum_{\substack{h\\q\text{-rim hook}\\\text{of }\lambda}}(-1)^{ll(h)}\chi^{\lambda\backslash h}(\pi\backslash q).
\end{equation*}
\end{lem}
\begin{proof} See \cite[§9]{Na1} or \cite[Theorem 4.10.2]{Sa1}.\end{proof}

$\chi^{\lambda}(1)$, the dimension of the irreducible representation associated with $\lambda$, can be computed via the \textit{hook formula.} 

\begin{defn}
Let $\lambda\vdash n$ be a partition. Denote by
\begin{equation*}
H_{i,j}(\lambda):=
\begin{cases}
\{(i,j')\in\lambda:~j'\geq j\}\cup\{(i',j)\in\lambda:~i'\geq i\}, &\text{if }(i,j)\in\lambda\\
\emptyset, &\text{if }(i,j)\notin\lambda.
\end{cases}
\end{equation*}
the \emph{hook} of the box $(i,j)$ in $\lambda$. Here we write $(i,j)\in \lambda$ to indicate that $(i,j)$ is a box in the Ferrers diagram of $\lambda$.
\end{defn}

\begin{lem}[Hook Formula]
\label{2-3-3}
Let $\lambda\vdash n$ be a partition. Then we obtain
\begin{equation*}
\chi^{\lambda}(1)=\frac{n!}{\prod\limits_{(i,j)\in \lambda}|H_{i,j}(\lambda)|}.
\end{equation*}
\end{lem}
\begin{proof} See \cite[Theorem 1]{Fr1} or \cite[Theorem 3.10.2]{Sa1}.\end{proof}\par

Finally, we will use the following estimate, due to T. W. Müller and J.-C. Schlage-Puchta \cite[Theorem 1]{Mu1}. 

\begin{lem}
\label{2-3-4}
Let $\chi\in\operatorname{Irr}(S_n)$ be an irreducible character, let $n$ be sufficiently large, and let $\sigma\in S_n$ be an element with $k$ fixed points. Then we have
\begin{equation*}
|\chi(\sigma)|\leq \chi(1)^{1-\frac{\log(n/k)}{32 \log n}}.
\end{equation*}
\end{lem}

\section{Character Estimates}

In this section we will count permutations $\sigma$ which fulfill $[\pi,\sigma]\in C$ for a conjugacy class $C$ of $S_m$ and a fixed $\pi\in S_m$. The proof of Theorem \ref{2-2-3} is based on the outcomes established below.

\begin{prop}
\label{2-4-1}
Let $p$ denote a prime with $\frac{1}{2}n<p\leq \frac{5}{8}n$. Let $\pi\in A_n$ be a fixed permutation with cycle type 
\begin{equation*}
\lambda =
\begin{cases}
(p,~n-p-2,~1,~1), &\text{if } ~n\equiv 0(2)\\
(p,~n-p-2,~2), &\text{if } ~n\equiv 1(2).
\end{cases}
\end{equation*}
In addition suppose that $C\subset A_n$ is a conjugacy class of $S_n$ such that
\begin{equation*}
f(C)\leq \delta\cdot n\qquad \text{ for } \delta\in(0,1).
\end{equation*}
Then we have for all sufficiently large $n$
\begin{equation*}
\#\{\sigma\in A_n:~[\pi,\sigma]\in C \}=|C|(1+\mathcal{O}(\delta)).
\end{equation*}
\end{prop}

\begin{proof} Due to the prime number theorem, there exists a prime $p$ with $\frac{1}{2}n<p\leq \frac{5}{8}n$ for all sufficiently large $n$. Denote by $K_{\pi}$ the conjugacy class of $\pi$ in $S_n$. Because of the cycle type of $\pi$, $K_{\pi}$ remains a single conjugacy class inside $A_n$ (cf. \cite[pp. 16f.]{Wi1}). Taking into account that $\chi(\tau^{-1})=\overline{\chi(\tau)}=\chi(\tau)$, it follows with Lemma \ref{2-3-1}

\begin{align}
\notag
\#\{\sigma\in A_n:~[\pi,\sigma]\in C \}&=\sum_{\substack{(x,y)\in K_{\pi}\times C\\ xy=\pi}}\#\{\sigma\in A_n:~\sigma^{-1}\pi\sigma=x\}\\\notag
&=\frac{n!}{2|K_{\pi}|}\#\{(x,y)\in K_{\pi}\times C:~ xy=\pi\}\\
&=\frac{|C|}{2}\sum_{\chi\in\operatorname{Irr}(S_n)}\frac{\chi^2(\pi)\cdot\chi(C)}{\chi(1)}.
\label{N-2-4-1}
\end{align}
\noindent
The contribution of the linear characters to (\ref{N-2-4-1}) is $|C|$. This is the expected main term. We shall show that the remaining characters give terms, which can be absorbed into the error term.\par\smallskip

i) We state: \textit{Let $\lambda\vdash n$ with $\chi^{\lambda}(\pi)\neq 0$. Then the Ferrers diagram of $\lambda$ has at most two boxes which do not belong to $H_{1,1}(\lambda)\cup H_{2,2}(\lambda)$. Moreover, let $h$ be a $p$-rim hook of $\lambda$ with $\chi^{\lambda\backslash h}(\pi\backslash p)\neq 0$ or a $(n-p-2)$-rim hook of $\lambda$ with $\chi^{\lambda\backslash h}(\pi\backslash (n-p-2))\neq 0$. Then the Ferrers diagram of $\lambda\backslash h$ has at most two boxes which do not belong to $H_{1,1}(\lambda\backslash h)$. In particular, we have $|H_{1,1}(\lambda)|\geq p$.}\\
These results follow from the special cycle type of $\pi$ and the Murnaghan-Nakayama rule (Lemma \ref{2-3-2}).\par

ii) \textit{If $\lambda\vdash n$, then we have $|\chi^{\lambda}(\pi)|\leq 2$.}\\
You can see this inequality as follows: If $\chi^{\lambda}(\pi)=0$ there is nothing to show. Hence, let us assume that $\chi^{\lambda}(\pi)\neq 0$. Denote 
\begin{equation*}
R:=\{h:~ h\text{ is a }p\text{-rim hook of }\lambda\text{ with }\chi^{\lambda\backslash h}(\pi\backslash p)\neq 0\}.
\end{equation*}
We have $R\neq \emptyset$ as a result of the Murnaghan-Nakayama rule. Due to $|H_{1,1}(\lambda)|\geq p>\frac{1}{2}n$ (see i)) a rim hook $h\in R$ contains boxes of the hook $H_{1,1}(\lambda)$. Hence $|R|\leq 2$. Now let $h_1\in R$. Because of i) there is exactly one $(n-p-2)$-rim hook $h_2$ of $\lambda\backslash h_1$. Obviously, there exists exactly one $2$-rim hook of $(\lambda\backslash h_1)\backslash h_2$ and exactly one $1$-rim hook of $(\lambda\backslash h_1)\backslash h_2$.\\
The Murnaghan-Nakayama rule yields that $|\chi^{\lambda}(\pi)|$ is equal to
\begin{align*}
\begin{cases}
\Bigl|\sum\limits_{\substack{h_1\\p\text{-rim hook}\\ \text{of }\lambda}}~~\sum\limits_{\substack{h_2\\(n-p-2)\text{-rim hook}\\\text{of }\lambda\backslash h_1}}(-1)^{ll(h_1)+ll(h_2)}\Bigr|, & \text{if } n\equiv 0(2)\\
\Bigl|\sum\limits_{\substack{h_1\\p\text{-rim hook}\\\text{of }\lambda}}~~\sum\limits_{\substack{h_2\\(n-p-2)\text{-rim hook}\\\text{of }\lambda\backslash h_1}}~\sum\limits_{\substack{h_3\\2\text{-rim hook}\\\text{of }(\lambda\backslash h_1)\backslash h_2}}(-1)^{ll(h_1)+ll(h_2)+ll(h_3)}\Bigr|, & \text{if } n\equiv 1(2).
\end{cases}
\end{align*}
Therefore, it follows that $|\chi^{\lambda}(\pi)|\leq 2$.

iii) Let $\lambda\vdash n.$ Define
\begin{align*}
m_1:&=(\text{number of boxes in the first column of the Ferrers diagram of }\lambda)-1\\
m'_1:&=(\text{number of boxes in the first row of the Ferrers diagram of }\lambda)-1\\
m_2:&=(\text{number of boxes in the second column of the Ferrers diagram of }\lambda)-2\\
m'_2:&=(\text{number of boxes in the second row of the Ferrers diagram of }\lambda)-2.
\end{align*}
We claim the following: \textit{Let $\lambda\vdash n$ such that the Ferrers diagram of $\lambda$ contains the box $(2,2)$ and $\chi^{\lambda}(\pi)\neq 0$. Then we have}
\begin{equation*}
\chi^{\lambda}(1)\geq\frac{1}{24n^2}\binom{n}{|H_{1,1}(\lambda)|}\binom{|H_{1,1}(\lambda)|-1}{m_1}
\end{equation*}
\textit{and}
\begin{equation*}
\chi^{\lambda}(1)\geq\frac{1}{24n^2}\binom{n}{|H_{1,1}(\lambda)|}\binom{|H_{1,1}(\lambda)|-1}{m'_1}.
\end{equation*}
\textit{Furthermore, if $|H_{1,1}(\lambda)|\geq p+3$ then there exists an $i\in\{0,1,2\}$ such that}
\begin{equation*}
m_1=m_2+|H_{1,1}(\lambda)|-p-i\qquad \textit{or}\qquad m'_1=m'_2+|H_{1,1}(\lambda)|-p-i.
\end{equation*}
\textit{Proof.} Due to $m_1+m'_1=|H_{1,1}(\lambda)|-1$ the first two formulas are equivalent. In view of i) we have the three cases $|H_{3,3}(\lambda)|=0,~|H_{3,3}(\lambda)|=1$ and $|H_{3,3}(\lambda)|=2$. In each case, the hook formula (Lemma \ref{2-3-3}) yields
\begin{equation*}
\chi^{\lambda}(1)\geq\frac{1}{24n^2}\binom{n}{|H_{1,1}(\lambda)|}\binom{|H_{1,1}(\lambda)|-1}{m_1}\binom{n-|H_{1,1}(\lambda)|-1}{m_2}.
\end{equation*}
Hence our first result follows. We now turn to the second statement: It follows from the Murnaghan-Nakayama rule that there is a $(n-p-2)$-rim hook $h$ of $\lambda$ such that $\chi^{\lambda\backslash h}(\pi\backslash (n-p-2))\neq 0$. By assumption we have $|H_{1,1}(\lambda)|\geq p+3$. Therefore, $h$ contains either boxes from the first column or the first row of the Ferrers diagram of $\lambda$ - but not both at the same time. Without loss of generality we assume that the first occurs. With regard to i) we have $|H_{2,2}(\lambda\backslash h)|\in\{0,1,2\}$. If $|H_{2,2}(\lambda\backslash h)|=1$ we obtain $n-p-2=n-|H_{1,1}(\lambda)|-1+m_1-m_2$ and hence $m_1=m_2+|H_{1,1}(\lambda)|-p-1$. The other cases are analog. This proves our claim.\par

iv) We now observe: \textit{Let $\lambda\vdash n,~\lambda\neq (n),~\lambda\neq(1^n)$ such that $\chi^{\lambda}(\pi)\neq 0$ and the Ferrers diagram of $\lambda$ does not contain the box $(2,2)$.\\
1) If $\lambda\in\{(n-1,1), (2,1^{n-2})\}$ we obtain $\chi^{\lambda}(1)=n-1$.\\
2) If $\lambda\in\{(p+1,1^{n-p-1}), (n-p,1^{p})\}$ we have $\chi^{\lambda}(1)=\binom{n-1}{p}$.\\
3) If $\lambda\in\{(n-p-1,1^{p+1}), (p+2,1^{n-p-2})\}$ it follows $\chi^{\lambda}(1)=\binom{n-1}{p+1}$.\\
The cases specified above are the only ones that can occur.}\\
The computations for $\chi^{\lambda}(1)$ follow immediately from the hook formula. Furthermore, the Murnaghan-Nakayama rule yields: There exists a $p$-rim hook $h_1$ of $\lambda$ such that there is a $(n-p-2)$-rim hook of $\lambda\backslash h_1$. This proves the last statement.\par

v) \textit{Let $\lambda\vdash n$ such that $\lambda\notin\{(n), (1^n), (n-1,1), (2,1^{n-2})\}$ and $\chi^{\lambda}(\pi)\neq 0$. Then we obtain for all sufficiently large $n$}
\begin{equation*}
\chi^{\lambda}(1)\geq \frac{1}{24n^2}\cdot 2^{n/8}.
\end{equation*}
You can see the estimate as follows: If the Ferrers diagram of $\lambda$ does not contain the box $(2,2)$, then it follows from iv) 
\begin{equation*}
\chi^{\lambda}(1)\geq\binom{n-1}{p+1}\geq \left(\frac{n-1}{p+1}\right)^{p+1}\geq 2^{n/8}.
\end{equation*}
Now, let us assume that the Ferrers diagram of $\lambda$ contains the box $(2,2)$.\\
Case 1: $|H_{1,1}(\lambda)|\leq\frac{7}{8}n$\\
By applying iii) and i), we get 
\begin{equation*}
\chi^{\lambda}(1)\geq\frac{1}{24n^2}\binom{n}{|H_{1,1}(\lambda)|}\geq\frac{1}{24n^2}\left(\frac{n}{n-|H_{1,1}(\lambda)|}\right)^{n-|H_{1,1}(\lambda)|}\geq \frac{1}{24n^2}\cdot 2^{n/8}.
\end{equation*}
Case 2: $|H_{1,1}(\lambda)|\geq\frac{7}{8}n$ \\
Without loss of generality there is an  $i\in\{0,1,2\}$ such that $m_1=m_2+|H_{1,1}(\lambda)|-p-i~$ (see iii)). Because of $0\leq m_2\leq n-|H_{1,1}(\lambda)|-1$ we realize that $\frac{1}{8}n\leq m_1\leq |H_{1,1}(\lambda)|-1-\frac{1}{8}n.$ If $\frac{1}{8}n\leq m_1\leq \frac{1}{2}(|H_{1,1}(\lambda)|-1)$, we obtain with iii)
\begin{equation*}
\chi^{\lambda}(1)\geq \frac{1}{24n^2}\binom{|H_{1,1}(\lambda)|-1}{m_1}\geq \frac{1}{24n^2}\left(\frac{|H_{1,1}(\lambda)|-1}{m_1}\right)^{m_1}\geq \frac{1}{24n^2}\cdot 2^{n/8}.
\end{equation*}
The case $\frac{1}{2}(|H_{1,1}(\lambda)|-1)\leq m_1\leq |H_{1,1}(\lambda)|-1-\frac{1}{8}n$ is analog. Hence our result follows.\par

vi) \textit{We have $\chi^{(n-1,1)}(C)=f(C)-1$ and $|\chi^{(n-1,1)}(\sigma)|=|\chi^{(2,1^{n-2})}(\sigma)|$.}\\
See \cite[Example 2.3.8]{Sa1} and \cite[Lemma 3.6.10]{Ce2}.\par\smallskip

We return to our original problem. In view of (\ref{N-2-4-1}) it remains to show
\begin{equation*}
\sum_{\substack{\chi\in\operatorname{Irr}(S_n)\\\chi(1)\neq 1}}\frac{\chi^2(\pi)\cdot\chi(C)}{\chi(1)}=\mathcal{O}(\delta).
\end{equation*}
Let $P_1:=\{\lambda\vdash n:~\lambda\notin\{(n), (1^n), (n-1,1), (2,1^{n-2})\}\}$ and $P_2:=\{(n-1,1), \linebreak (2,1^{n-2})\}$. By applying Lemma \ref{2-3-4} and ii), we get

\begin{equation*}
\sum_{\substack{\chi\in\operatorname{Irr}(S_n)\\\chi(1)\neq 1}}\frac{\chi^2(\pi)\cdot\chi(C)}{\chi(1)}\ll\sum_{\substack{\lambda\in P_1\\\chi^{\lambda}(\pi)\neq 0}}\chi^{\lambda}(1)^{\frac{\log\delta}{32\log n}}+\sum_{\lambda\in P_2}\frac{|\chi^{\lambda}(C)|}{\chi^{\lambda}(1)}.
\end{equation*}
\noindent
We consider separately the two sums. i), iv) and v) yield
\begin{equation*}
\sum_{\substack{\lambda\in P_1\\\chi^{\lambda}(\pi)\neq 0}}\chi^{\lambda}(1)^{\frac{\log\delta}{32\log n}}\ll n^3\cdot\left(\frac{1}{24n^2}\cdot 2^{n/8}\right)^{\frac{\log \delta}{32 \log n}}\ll \exp\left(\frac{\log \delta}{400}\frac{n}{\log n}\right),
\end{equation*}
which is far smaller than necessary. Due to iv) and vi), we have
\begin{equation*}
\sum_{\lambda\in P_2}\frac{|\chi^{\lambda}(C)|}{\chi^{\lambda}(1)}\leq 2\frac{|f(C)-1|}{n-1}\ll \delta.
\end{equation*}
This completes the proof.\end{proof}\par

\noindent
Below we will state results similar to Proposition \ref{2-4-1} for different cycle types of $\pi$.

\begin{prop}
\label{2-4-2}
Let $p$ denote a prime with $\frac{1}{2}n<p\leq \frac{5}{8}n$. Let $\pi\in S_{n-1}$ be a fixed permutation with cycle type $\lambda =(p,~n-p-2,~1)$. In addition suppose that $C\subset A_{n-1}$ is a conjugacy class of $S_{n-1}$ such that
\begin{equation*}
f(C)\leq \delta\cdot (n-1)\qquad \text{ for } \delta\in(0,1).
\end{equation*}
Then we have for all sufficiently large $n$ 
\begin{equation*}
\#\{\sigma\in S_{n-1}:~[\pi,\sigma]\in C \}=|C|\biggl(2+\mathcal{O}\Bigl(\exp \Bigl(\frac{\log\delta}{400}\frac{n}{\log n}\Bigr)\Bigr)\biggr).
\end{equation*}
\end{prop}
\noindent
The proof is similar to that of Proposition \ref{2-4-1}.

\begin{prop}
\label{2-4-3}
Let $p$ denote a prime with $\frac{1}{2}n<p\leq \frac{5}{8}n$. Let $\pi\in S_{n-2}$ be a fixed permutation with cycle type $\lambda =(p,~n-p-2)$. In addition suppose that $C\subset A_{n-2}$ is a conjugacy class of $S_{n-2}$ such that
\begin{equation*}
f(C)\leq \delta\cdot (n-2)\qquad \text{ for } \delta\in(0,\tfrac{2}{3}].
\end{equation*}
Then we have for all sufficiently large $n$
\begin{equation*}
\#\{\sigma\in S_{n-2}:~[\pi,\sigma]\in C \}=|C|(2+\mathcal{O}(\delta^{1/33})).
\end{equation*}
\end{prop}
\begin{proof} Apart from an estimate, the arguments are analogous to those in the proof of Proposition \ref{2-4-1}. We will discuss the missing detail below.\end{proof}\par

\begin{prop}
\label{2-4-4}
Let $\pi\in S_m$ be a fixed permutation with cycle type $\lambda\in\{(m),\linebreak(m-1,1), (m-2,2), (m-2,1,1)\}$. In addition suppose that $C\subset A_m$ is a conjugacy class of $S_m$ such that
\begin{equation*}
f(C)\leq \delta\cdot m\qquad \text{ for } \delta\in(0,\tfrac{2}{3}].
\end{equation*}
Then we have for all sufficiently large $m$
\begin{equation*}
\#\{\sigma\in S_m:~[\pi,\sigma]\in C \}=|C|(2+\mathcal{O}(\delta^{1/33})).
\end{equation*}
\end{prop}

\begin{proof} We sketch the argument for $\lambda=(m)$. The other cases are similar. Because of the cycle type of $\pi$, the Murnaghan-Nakayama rule and the hook formula yield:\\
\textit{Let $\lambda\vdash m$ with $\chi^{\lambda}(\pi)\neq 0$. Then we get:\\
i) The Ferrers diagram of $\lambda$ does not contain the box $(2,2)$.\\
ii) ~~$|\chi^{\lambda}(\pi)|= 1$.\\
iii) ~$\chi^{\lambda}(1)=\binom{m-1}{m_1}$, where $m_1$ is defined as in the proof of Proposition \ref{2-4-1}.}\\
Analogous to (\ref{N-2-4-1}), we have
\begin{align*}
\#\{\sigma\in S_m:~[\pi,\sigma]\in C \}&=|C|\sum_{\chi\in\operatorname{Irr}(S_m)}\frac{\chi^2(\pi)\cdot\chi(C)}{\chi(1)}\\
&=|C|\Biggr(2+\sum_{\substack{\chi\in\operatorname{Irr}(S_m)\\\chi(1)\neq 1}}\frac{\chi^2(\pi)\cdot\chi(C)}{\chi(1)}\Biggl).
\end{align*}
We now use Lemma \ref{2-3-4} and i)-iii) to estimate the error term:
\begin{align*}
\Bigr|\sum_{\substack{\chi\in\operatorname{Irr}(S_m)\\\chi(1)\neq 1}}\frac{\chi^2(\pi)\cdot\chi(C)}{\chi(1)}\Bigl|~&\leq~\sum_{m_1=1}^{m-2}\binom{m-1}{m_1}^{\frac{\log\delta}{32\log m}}\\
&\ll\sum_{1\leq m_1\leq\frac{1}{2}(m-1)}\left(\frac{m-1}{m_1}\right)^{\frac{\log\delta}{32\log m}m_1}.
\end{align*}
We split the summation over $m_1$ in two ranges and consider separately the subsums. We obtain
\begin{equation*}
\sum_{1\leq m_1\leq m^{1/34}}\left(\frac{m-1}{m_1}\right)^{\frac{\log\delta}{32\log m}m_1}\leq\sum_{m_1=1}^{\infty}\delta^{\frac{m_1}{33}}=\frac{\delta^{1/33}}{1-\delta^{1/33}}\ll\delta^{1/33}
\end{equation*}
and
\begin{equation*}
\sum_{m^{1/34}\leq m_1\leq \frac{1}{2}(m-1)}\left(\frac{m-1}{m_1}\right)^{\frac{\log\delta}{32\log m}m_1}\leq m\cdot\exp\left(\log \delta\cdot\frac{\log 2}{32}\frac{m^{1/34}}{\log m}\right),
\end{equation*}
which is far smaller than necessary.\end{proof}\par

\section{Completion of the Proof}

Choose $\pi\in A_n$ and $C\subset A_n$ as in Theorem \ref{2-2-3}. We shall show that counting permutations $\sigma\in A_n$ such that $[\pi,\sigma]\in C$ and $\langle\pi,\sigma\rangle=A_n$ can be reduced to the problems dealt with in the last section.\\
In this context, the following result on permutation groups due to C. Jordan (cf. \cite[Theorem 3.3E]{Di1} or \cite[Theorem 13.9]{Wi2}) is very useful.

\begin{lem}
\label{2-5-1}
Let $G$ be a primitive subgroup of $S_n$. Suppose that $G$ contains at least one permutation which is a $p$-cycle for a prime $p\leq n-3$. Then either $G=S_n$ or $G=A_n$ holds.
\end{lem}
\noindent
We now show how this Lemma can be applied to our problem.

\begin{lem}
\label{2-5-2}
Let $G$ be a transitive subgroup of $S_n$. Suppose that there exists a $\pi\in G$ containing a $p$-cycle for a prime $p$ with $\frac{1}{2}n<p\leq n-3$. Then either $G=S_n$ or $G=A_n$ holds.
\end{lem}

\begin{proof} i) We claim: $G$ \textit{is primitive.}\\
This can be seen as follows: Let us assume $G$ is imprimitive. Then there exist blocks $Y_1,\ldots,Y_k$ for $G$ such that $Y_1,\ldots,Y_k$ form a partition of $\{1,\ldots,n\}$, $2\leq |Y_1|\leq n-1$ and $|Y_1|=|Y_i|$ for all $i=2,\ldots,k$. Let $a$ be an element contained in the $p$-cycle of $\pi$. Then we have $a\in Y_{l_0},~\pi(a)\in Y_{l_1},\ldots,\pi^{p-1}(a)\in Y_{l_{p-1}}$ for pairwise distinct $l_i\in\{1,\ldots,k\}$. Hence, we get $k\geq p$. This yields the contradiction $n=|Y_1|k\geq 2p>n.$\\
ii) Let $d:=\operatorname{lcm}\{m:~\pi\text{ contains a cycle of length } m<p\}$. Obviously, $\pi^d\in G$ is a $p$-cycle. In view of i) and Lemma \ref{2-5-1} we can conclude: $G=S_n$ or $G=A_n$.\end{proof}\par
\noindent
The last Lemma yields immediately
\begin{cor}
\label{2-5-3}
Let $p$ denote a prime with $\frac{1}{2}n<p\leq n-3$. Suppose that $\pi\in A_n$ contains a $p$-cycle and let $C$ be a conjugacy class of $S_n$. Then we get
\begin{align*}
\#&\{\sigma\in A_n:~[\pi,\sigma]\in C \wedge \langle\pi,\sigma\rangle=A_n\}\\
&=\#\{\sigma\in A_n:~[\pi,\sigma]\in C\}-\#\{\sigma\in A_n:~[\pi,\sigma]\in C \wedge (\langle\pi,\sigma\rangle\text{ is not transitive})\}.
\end{align*}
\end{cor}
\noindent
In the above equation we are able to compute the first term on the right hand side (see Proposition \ref{2-4-1}). In order to calculate the second one, we have to analyze the condition $\langle\pi,\sigma\rangle$ \textit{is not transitive.}

\begin{defn}
For a cycle $\tau=(i_1,\ldots,i_m)$ define $M(\tau):=\{i_1,\ldots,i_m\}$.\\
Let $\pi_1\cdot\ldots\cdot\pi_k$ be the cycle decomposition for $\pi\in S_n$ (including $1$-cycles) and set $\Omega:=\{1,2,\ldots,n\}$. Define $\Gamma_H:=\bigcup_{i\in H}M(\pi_i)\subset \Omega$ when $H\subset \{1,\ldots,k\}$.\\
For a non-empty set $\Omega$ let $Sym(\Omega)$ denote the symmetric group on $\Omega$. When $\Gamma\subset\Omega$ we identify $Sym(\Gamma)$ with the subgroup of $Sym(\Omega)$ consisting of all $\tau\in Sym(\Omega)$ with $\operatorname{supp}(\tau)\subset\Gamma$.
\end{defn}

\begin{lem}
\label{2-5-4}
Let $\pi_1\cdot\ldots\cdot\pi_k$ be the cycle decomposition for $\pi\in S_n$ (including $1$-cycles) and let $\sigma\in S_n$. Then $\langle\pi,\sigma\rangle$ is not transitive if and only if there exists a non-empty proper subset $H$ of $K:=\{1,\ldots,k\}$ and a tuple $(\sigma',\sigma'')\in Sym(\Gamma_H)\times Sym(\Gamma_{K\backslash H})$ such that $\sigma=\sigma'\sigma''$.
\end{lem}

\begin{proof} "$\Leftarrow$": Obvious.\\
"$\Rightarrow$": This implication is the result of the following stronger statement: \textit{Let $\sigma_1\cdot\ldots\cdot\sigma_l$ be the cycle decomposition for $\sigma$ (including $1$-cycles). Let $s,t\in\{1,\ldots,n\}$ such that $\tau(s)\neq t$ for all $\tau\in\langle\pi,\sigma\rangle$. Then: For all $m$ with $1\leq m\leq l-1$ there exists non-empty subsets $H_m,H'_m\subset K$ and a subset $L_m\subset L:=\{1,\ldots,l\}$ such that\\
a) $\forall j\in L_m: ~~M(\sigma_j)\subset\bigcup_{i\in H_m}M(\pi_i) ~~\vee ~~ M(\sigma_j)\subset\bigcup_{i\in H'_m}M(\pi_i)$\\
b) $|L_m|\geq m+1$\\
c) $H_m\cap H'_m=\emptyset$\\
d) $L_m\neq L\Rightarrow~ (\forall i\in H_m~\exists\tau\in\langle\pi,\sigma\rangle:~\tau(s)\in M(\pi_i))$\\
e) $L_m\neq L\Rightarrow~ (\forall i\in H'_m~\exists\tau\in\langle\pi,\sigma\rangle:~\tau(t)\in M(\pi_i))$.}\\
We prove this claim by induction: If $m=1$ we have $s\in M(\sigma_j)$ and $t\in M(\sigma_{j'})$ for $j\neq j'$. Set $L_1:=\{j,j'\}$, $H_1:=\{i\in K:~M(\sigma_j)\cap M(\pi_i)\neq\emptyset\}$ and $H'_1:=\{i\in K:~M(\sigma_{j'})\cap M(\pi_i)\neq\emptyset\}$. Then a) - e) follow.\\
Now let $1\leq m\leq l-2$ and suppose that the statement is true for $m$. The case $L_m=L$ is trivial. Hence, let us assume that $L_m\subsetneq L$.\\
Case 1: $\exists j_0\in L\setminus L_m:~~M(\sigma_{j_0})\cap\bigcup_{i\in H_m}M(\pi_i)\neq\emptyset ~~\vee ~~ M(\sigma_{j_0})\cap\bigcup_{i\in H'_m}M(\pi_i)\neq\emptyset$.\\
Without loss of generality assume that $M(\sigma_{j_0})\cap\bigcup_{i\in H_m}M(\pi_i)\neq\emptyset$. Define $L_{m+1}:=L_m\cup\{j_0\}$, $H_{m+1}:=H_m\cup\{i\in K:~M(\sigma_{j_0})\cap M(\pi_i)\}\neq\emptyset$ and $H'_{m+1}:=H'_m$. Then a) - e) result from the induction hypothesis and some easy computations.\\
Case 2: $\forall j\in L\setminus L_m:~~M(\sigma_j)\cap\bigcup_{i\in H_m}M(\pi_i)=\emptyset ~~\wedge ~~ M(\sigma_j)\cap\bigcup_{i\in H'_m}M(\pi_i)=\emptyset$.\\
Set $L_{m+1}:=L,~ H_{m+1}:=K\setminus H'_m$ and $H_{m+1}:=H'_m$. Obviously a) - e) are fulfilled.\end{proof}\par

\begin{defn}
Denote by $[S_n]$ the set of conjugacy classes of $S_n$.\\
Let $C\in[S_n]$, $C_1\in[S_k]$ and $C_2\in[S_{n-k}]$. Denote $m_l,~s_l$ and $t_l$ the number of $l$-cycles in $C,~C_1$ and $C_2$, respectively. We say $~C=C_1C_2~$ iff $~m_l=s_l+t_l$ for all $l=1,\ldots,n$.
\end{defn}

Finally, we draw our attention to $\sum |C_1||C_2|$, where the sum runs over all $(C_1,C_2)\in [S_k]\times[S_{n-k}]$ satisfying $C_1C_2=C$ for a fixed conjugacy class $C$ of $S_n$. In this context, the following Lemma is very useful.

\begin{lem}
\label{2-5-5}
Let $C$ be a conjugacy class of $S_n$ with cycle type $(1^{m_1},2^{m_2},\ldots,n^{m_n})$, where $m_l$ is the number of $l$-cycles in $C$. Then we obtain
\begin{equation*}
|C|=\frac{n!}{1^{m_1}m_1!\cdot2^{m_2}m_2!\cdot\ldots\cdot n^{m_n}m_n!}.
\end{equation*}
\end{lem}
\begin{proof} See \cite[Proposition 1.1.1]{Sa1}.\end{proof}

\noindent
We now observe:

\begin{lem}
\label{2-5-6}
Suppose that $C$ is a conjugacy class of $S_n$ such that $f(C)\leq \delta\cdot n$ for $\delta\in(0,\tfrac{1}{4}]$ and let $n$ be sufficiently large.\\
1) Then we have for $\frac{1}{2}n<k\leq\frac{5}{8}n$
\begin{equation*}
\sum_{\substack{C_1\in [S_k]\\C_2\in[S_{n-k}]\\C=C_1C_2}}|C_1||C_2|\leq\exp\left(-\frac{3}{50}n\right)|C|.
\end{equation*}
2) Furthermore, we obtain
\begin{equation*}
\sum_{\substack{C_1\in [S_{n-1}]\\C_2\in[S_1]\\C=C_1C_2}}|C_1|\leq\delta |C|\qquad\text{and}\qquad\sum_{\substack{C_1\in [S_{n-2}]\\ C_2\in [S_2]\\C=C_1C_2}}|C_1|\leq 2\delta^2|C|.
\end{equation*}
\end{lem}

\begin{proof} Denote $(1^{m_1},2^{m_2},\ldots,n^{m_n})$ the cycle type of $C$.\\
We begin with 1): If $C$ contains a $l$-cycle with $l>k$, then the considered sum is equal to $0$. Therefore, we assume that $C$ has cycle type $(1^{m_1},2^{m_2},\ldots,k^{m_k},(k+1)^0,\ldots n^0)$. By computing the cardinality of $C$ and of conjugacy classes $C_1\in [S_k]$, $C_2\in [S_{n-k}]$ with $C=C_1C_2$ via Lemma \ref{2-5-5}, we obtain
\begin{align*}
\sum_{\substack{C_1\in [S_k]\\C_2\in [S_{n-k}]\\C=C_1C_2}}|C_1||C_2|~&\leq~|C|\binom{n}{k}^{-1}\!\sum_{\substack{(i_1,\ldots ,i_{n-k})\\0\leq i_j\leq m_j}}\prod_{l=1}^{n-k}\binom{m_l}{i_l}\\
&=~|C|\binom{n}{k}^{-1}2^{(m_1+\ldots +m_{n-k})}.
\end{align*}
The Stirling formula yields the inequality $\binom{n}{k}\geq \left(\frac{41}{25}\right)^n$ for $\frac{1}{2}n<k\leq \frac{5}{8}n$. In addition, we have $\sum_{i=1}^nm_i\leq f(C)+\frac{n-f(C)}{2}\leq \frac{1}{2}n(1+\delta)\leq \frac{5}{8}n$. Hence, our estimate follows.\\
We now show the second statement: Let $i\in\{1,2\}$. We again apply Lemma \ref{2-5-5}: By calculating the cardinality of $C$ and of conjugacy classes $C_1\in[S_{n-i}]$ such that there exists $C_2\in[S_i]$ with $C=C_1C_2$, we get
\begin{equation*}
\sum_{\substack{C_1\in [S_{n-1}]\\C_2\in[S_1]\\C=C_1C_2}}|C_1|\leq |C|\frac{m_1}{n}\leq\delta |C|
\end{equation*}
and
\begin{align*}
\sum_{\substack{C_1\in [S_{n-2}]\\C_2\in [S_2]\\C=C_1C_2}}|C_1|&\leq |C|\frac{1}{n(n-1)}\Bigl(m_1(m_1-1)+2m_2\Bigr)\\
&\leq |C|\left(\delta^2+\frac{1}{n-1}\right)\\
&\leq 2\delta^2|C|.
\end{align*}
This proves our claim.\end{proof}

Now we are able to give the
\begin{proof} [Proof of Theorem \ref{2-2-3}]
Let $\pi_1\cdot\ldots\cdot\pi_k$ be the cycle decomposition of $\pi$ (including $1$-cycles) and set $K:=\{1,\ldots,k\}$. Then Lemma \ref{2-5-4} yields
\begin{align*}
B:&=\#\{\sigma\in A_n:~[\pi,\sigma]\in C \wedge (\langle\pi,\sigma\rangle\text{ is not transitive})\}\\
&\leq \sum_{\substack{\{H,\,K\backslash H\}\\H\subsetneq K,\,H\neq \emptyset}}\#\{(\sigma',\sigma'')\in Sym(\Gamma_H)\times Sym(\Gamma_{K\backslash H}):~[\pi,\sigma'\sigma'']\in C\}.
\end{align*}

Let $H$ be a proper non-empty subset of $K$. Define $\pi':=\prod_{i\in H}\pi_i$ and $\pi'':=\prod_{i\in K\backslash H}\pi_i$. If $(\sigma',\sigma'')\in Sym(\Gamma_H)\times Sym(\Gamma_{K\backslash H})$ we get $[\pi,\sigma'\sigma'']=[\pi',\sigma']\!\cdot\![\pi'',\sigma'']$. Obviously, we have $[\pi',\sigma']\in Sym(\Gamma_H)$ and $[\pi'',\sigma'']\in Sym(\Gamma_{K\backslash H})$. Therefore, $[\pi,\sigma'\sigma'']\in C$ holds if and only if there exists conjugacy classes $C_1\subset Sym(\Gamma_H)$ and $C_2\subset Sym(\Gamma_{K\backslash H})$ with $[\pi',\sigma']\in C_1$, $[\pi'',\sigma'']\in C_2$ and $C=C_1C_2$.\par\smallskip

Case 1: $n\equiv 1(2)$. By assumption we have $k=3$ and $K=\{1,2,3\}$. For $\lambda\vdash m$, let $\pi_{_{\lambda}}\in S_m$ be the a fixed permutation with cycle type $\lambda$. The above considerations yield
\begin{align*}
B&\leq \sum_{\substack{C_1\in [S_{n-2}]\\C_2 \in [S_2]\\C=C_1C_2}}\#\{\sigma\in S_{n-2}:~[\pi_{_{(p,n-p-2)}},\sigma]\in C_1\}\cdot\#\{\sigma\in S_2:~[\pi_{_{(2)}},\sigma]\in C_2\}\\
&~\quad+\hspace{-0.3cm}\sum_{\substack{C_1\in [S_{p}]\\C_2\in[S_{n-p}]\\C=C_1C_2}}\hspace{-0.3cm}\#\{\sigma\in S_{p}:~[\pi_{_{(p)}},\sigma]\in C_1\}\cdot\#\{\sigma\in S_{n-p}:~[\pi_{_{(n-p-2,2)}},\sigma]\in C_2\}\\
&~\quad+\hspace{-0.4cm}\sum_{\substack{C_1\in [S_{p+2}]\\C_2 \in[S_{n-p-2}]\\C=C_1C_2}}\hspace{-0.45cm}\#\{\sigma\in S_{p+2}:~[\pi_{_{(p,2)}},\sigma]\in C_1\}\cdot\#\{\sigma\in S_{n-p-2}:~[\pi_{_{(n-p-2)}},\sigma]\in C_2\}\\
&=\mathcal{O}\Biggl(\sum_{\substack{C_1\in [S_{n-2}]\\C_2\in [S_2]\\C=C_1C_2}}|C_1|+\sum_{\substack{C_1\in [S_{p}]\\C_2\in [S_{n-p}]\\C=C_1C_2}}|C_1||C_2|+\sum_{\substack{C_1\in [S_{p+2}]\\C_2 \in[S_{n-p-2}]\\C=C_1C_2}}|C_1||C_2|\Biggr)\\
&=\mathcal{O}(\delta^2|C|).
\end{align*}
For $i\in\{1,2,3\}$ we have $\frac{1}{2}n<p+i\leq\frac{5}{8}n$. In addition, $f(C_1)\leq f(C)$ is valid for $C=C_1C_2$. Therefore, the last but one equality follows from Proposition \ref{2-4-3} and Proposition \ref{2-4-4}. The last step results from Lemma \ref{2-5-6}.\par
Case 2: $n\equiv 0(2)$. With argumentation similar to case 1 we obtain $B=\mathcal{O}(\delta|C|)$.\par\smallskip
Hence, it follows from Corollary \ref{2-5-3} and Proposition \ref{2-4-1}
\begin{equation*}
\#\{\sigma\in A_n:~[\pi,\sigma]\in C \wedge \langle\pi,\sigma\rangle=A_n\}=|C|\left(1+\mathcal{O}(\delta)\right).
\end{equation*}
So we are done.\end{proof}\par

\par\bigskip\noindent
\textbf{Author information}\\
Stefan-Christoph Virchow, Institut für Mathematik, Universität Rostock\\
Ulmenstr. 69 Haus 3, 18057 Rostock, Germany\\
E-mail: stefan.virchow@uni-rostock.de


\begin{thebibliography}{15}

\bibitem{An1}
\textsc{G. E. Andrews, K. Eriksson:} \textit{Integer Partitions}, Cambridge University Press, Cambridge (2004).

\bibitem{Ba1}
\textsc{L. Babai, I. Pak:} Strong bias of group generators: an obstacle to the "product replacement algorithm", \textit{Proceedings of the eleventh annual ACM-SIAM symposium on Discrete algorithms} (2000), 627-635.

\bibitem{Ce2}
\textsc{T. Ceccherini-Silberstein, F. Scarabotti, F. Tolli:} \textit{Representation Theory of the Symmetric Group}, Cambridge University Press, Cambridge (2010).

\bibitem{Ce1}
\textsc{F. Celler, C. R. Leedham-Green, S. H. Murray, A. C. Niemeyer, E. A. O'Brien:} Generating random elements of a finite group, \textit{Comm. Algebra} \textbf{23} (1995), 4931-4948.

\bibitem{Co1}
\textsc{G. Cooperman, I. Pak:} The product replacement graph on generating triples of permutations, \textit{preprint} (2000), 1-13.

\bibitem{Cu1}
\textsc{C. W. Curtis, I. Reiner:} \textit{Methods of Representation Theory, Volume I}, Wiley, New York (1990).

\bibitem{Da1}
\textsc{C. David:} $T_3$-systems of finite simple groups, \textit{Rend. Sem. Math. Univ. Padova} \textbf{89} (1993), 19-27.

\bibitem{Di2}
\textsc{J. D. Dixon:} The probability of generating the symmetric group, \textit{Math. Z.} \textbf{110} (1969), 199-205.

\bibitem{Di1}
\textsc{J. D. Dixon, B. Mortimer:} \textit{Permutation Groups}, Springer, New York (1996).

\bibitem{Du1}
\textsc{M. J. Dunwoody:} On $T$-systems of groups, \textit{J. Austral. Math. Soc.} \textbf{3} (1963), 172-179.

\bibitem{Ev2}
\textsc{M. J. Evans:} \textit{Problems Concerning Generating Sets for Groups,} Ph.D. Thesis, University of Wales (1985). Available at \textit{www.bl.uk}

\bibitem{Ev1}
\textsc{M. J. Evans:} $T$-systems of certain finite simple groups, \textit{Math. Proc. Cambridge Philos. Soc.} \textbf{113} (1993), 9-22.

\bibitem{Fr1}
\textsc{J. S. Frame, G. de B. Robinson, R. M. Thrall:} The hook graphs of the symmetric group, \textit{Canad. J. Math.} \textbf{6} (1954), 316-324.

\bibitem{Ga1}
\textsc{A. Gamburd, I. Pak:} Expansion of product replacement graphs, \textit{Combinatorica} \textbf{26} (2006), 411-429.

\bibitem{Ga3}
\textsc{S. Garion:} Connectivity of the product replacement algorithm graph of $PSL(2,q)$, \textit{J. Group Theory} \textbf{11} (2008), 765-777.

\bibitem{Ga2}
\textsc{S. Garion, A. Shalev:} Commutator maps, measure preservation, and $T$-systems, \textit{Trans. Amer. Math. Soc.} \textbf{361} (2009), 4631-4651.

\bibitem{Gi1}
\textsc{R. Gilman:} Finite quotients of the automorphism group of a free group, \textit{Canad. J. Math.} \textbf{29} (1977), 541-551.

\bibitem{Gu1}
\textsc{R. Guralnick, I. Pak:} On a question of B. H. Neumann, \textit{Proc. Amer. Math. Soc.} \textbf{131} (2003), 2021-2025.

\bibitem{Ha1}
\textsc{G. H. Hardy, S. Ramanujan:} Asymptotic formul{\ae} in combinatory analysis, \textit{Proc. Lond. Math. Soc.} \textbf{17} (1918), 75-115.

\bibitem{Ke1}
\textsc{A. Kerber:} \textit{Algebraic Combinatorics Via Finite Group Actions}, BI-Wissenschaftsverlag, Mannheim-Wien-Zürich (1991).

\bibitem{Lu1}
\textsc{A. Lubotzky, I. Pak:} The product replacement algorithm and Kazhdan's property (T), \textit{J. Amer. Math. Soc.} \textbf{14} (2001), 347-363.

\bibitem{Mu1}
\textsc{T. W. Müller, J.-C. Schlage-Puchta:} Character theory of symmetric groups, subgroup growth of Fuchsian groups, and random walks, \textit{Adv. Math.} \textbf{213} (2007), 919-982.

\bibitem{Na1}
\textsc{T. Nakayama:} On some modular properties of irreducible representations of a symmetric group, I, \textit{Jap. J. Math.} \textbf{17} (1940), 165-184.

\bibitem{Ne2}
\textsc{B. H. Neumann:} On a question of Gaschütz, \textit{Arch. Math.} \textbf{7} (1956), 87-90.

\bibitem{Ne1}
\textsc{B. H. Neumann, H. Neumann:} Zwei Klassen charakteristischer Untergruppen und ihre Faktorgruppen, \textit{Math. Nachr.} \textbf{4} (1950), 106-125.

\bibitem{Ni1}
\textsc{J. Nielsen:} Die Isomorphismengruppe der freien Gruppen, \textit{Math. Ann.} \textbf{91} (1924), 169-209.

\bibitem{Pa1}
\textsc{I. Pak:} What do we know about the product replacement algorithm?, in: \textit{Groups and Computations III}, editors: \textsc{W. M. Kantor, A. Seress,} de Gruyter, Berlin-New York (2001), 301-347.

\bibitem{Sa1}
\textsc{B. E. Sagan:} \textit{The Symmetric Group}, Springer, New York (2001).

\bibitem{Sc1}
\textsc{J.-C. Schlage-Puchta:} Applications of character estimates to statistical problems for the symmetric group, \textit{Combinatorica} \textbf{32} (2012), 309-323.

\bibitem{Wi2}
\textsc{H. Wielandt:} \textit{Finite Permutation Groups}, Academic Press, New York (1964).

\bibitem{Wi1}
\textsc{R. A. Wilson:} \textit{The Finite Simple Groups}, Springer, London (2009).



\end{thebibliography}
\end{document}